\documentclass[a4paper]{amsart}
\usepackage{amssymb,amsmath}
\usepackage{txfonts}
\usepackage[dvips]{graphicx}
\newcommand{\R}{\mathbb{R}}
\newcommand{\Z}{\mathbb{Z}}
\newcommand{\C}{\mathbb{C}}
\newcommand{\co}{\colon\thinspace}
\providecommand{\coloneqq}{\mathrel{\mathop:}=}
\newcommand{\restr}{\ensuremath{\thinspace\vrule\thinspace}}
\theoremstyle{plain}
\newtheorem{theorem}{Theorem}[section]
\newtheorem{corollary}[theorem]{Corollary}
\newtheorem{proposition}[theorem]{Proposition}

\theoremstyle{definition}
\newtheorem{definition}[theorem]{Definition}
\newtheorem{remark}[theorem]{Remark}
\newenvironment{acknowledgements}{\bigskip\noindent\textbf{Acknowledgements.}}{}
\title{Knots and links of complex tangents}
\author{Naohiko Kasuya \and Masamichi Takase}
\address{Naohiko Kasuya: 
School of Social Informatics, 
Aoyama Gakuin University, 5-10-1 Fuchinobe, Chuo-ku, 
Sagamihara, Kanagawa 252-5258, Japan.}
\email{nkasuya@si.aoyama.ac.jp}%
\address{Masamichi Takase: 
Faculty of Science and Technology,
Seikei University, 3-3-1 Kichijoji-kitamachi, Musashino,
Tokyo 180-8633, Japan.} 
\email{mtakase@st.seikei.ac.jp}%
\date{}
\begin{document}\sloppy
\maketitle
\begin{abstract}
It is shown that every knot or link is the set of complex tangents of a
$3$-sphere smoothly embedded in the three-dimensional complex space. 
We show in fact that a one-dimensional submanifold of 
a closed orientable $3$-manifold 
can be realised as the set of complex tangents 
of a smooth embedding of the $3$-manifold 
into the three-dimensional complex space if and only if 
it represents the trivial integral 
homology class in the $3$-manifold. 
The proof involves a new application of 
singularity theory of differentiable maps. 
\end{abstract}
\section{Introduction}\label{sect:intro}
An immersion $f$ of a $C^\infty$-smooth manifold $M$ 
into the complex space $\C^n$ is said to be \textit{totally real}  
if $df_x(T_xM)\cap J(df_x(T_xM))=\{0\}$ for 
each point $x\in M$ and 
the complex structure $J$. 
If, on the contrary, 
$df_x(T_xM)$ contains a complex line, 
such a point $x$ is said to be 
\textit{a complex tangent}. 
Totally real immersions and embeddings 
have long been important topics in differential geometry (see e.\,g.\ \cite{MR787894,MR966952,MR880125,MR864505}).
The behaviour of complex tangents is 
also apparently interesting and has been extensively  
studied (see e.\,g.\ \cite{MR0200476,MR1360632,MR2928578,Ali_pre,MR1177310,MR0314066,MR800003}). 

In this paper we show that 
a $1$-dimensional submanifold $L$ 
of a closed orientable $3$-manifold $M^3$ 
can be realised as the set of complex tangents 
of a $C^\infty$-smooth embedding $M^3$ into $\C^3$ 
if and only if the homology class $[L]$ vanishes in $H_1(M^3;\Z)$. 
Ali M.~Elgindi has obtained, 
in his pioneering paper \cite{MR2928578}, 
a similar result mainly for a knot in the $3$-sphere $S^3$, 
namely in the case where $L$ is a single circle and 
$M^3=S^3$, 
in which, however, the embedding of $S^3$ into $\C^3$ 
ought to have a degenerate point and 
cannot be taken to be $C^\infty$-smooth 
(see also \cite{Ali_pre}). 
In his argument, 
Akbulut and King's result \cite{MR639356} 
has played a crucial role to relate 
the two seemingly unrelated objects 
--- geometry of complex tangents and topology of knots. 
Our approach is quite different; we employ instead 
Saeki's theorem on singularities of stable maps 
in the spirit of differential topology. 
This enables us to avoid 
dealing with the degeneracy, 
and to study knots and links in a general 
orientable $3$-manifold. 

A stable map between manifolds, 
which we will define later in terms of 
the conditions of local forms,  
is a $C^\infty$-smooth map which differs from neighbouring maps 
in the mapping space only by diffeomorphisms 
of the source and target manifolds. 
The notion of a stable map 
can be naturally regarded as a high-dimensional 
variant of a Morse function 
and has attracted attention as a tool 
to analyse the topology of a manifold. 
We especially focus on \textit{liftable} stable maps
from $3$-manifolds to the plane, 
that is, those stable maps which can factor through immersions into $\R^4$, 
and reveal that stable maps are useful to study the geometry 
of a real submanifold in a complex space. 

Our proof is not complicated; it contains 
two main ideas --- 
a refinement of Saeki's theorem claiming 
that any integrally null-homologous link is the singular set of a liftable 
stable map to the plane (Theorem~\ref{thm:ext}) 
and a gimmick to lift the stable map 
into an immersion in $\C^3$
whose complex tangents form the given link 
(Theorem~\ref{thm:imm}). 
At the final step, 
with the aid of a totally real version of Whitney's trick 
(Gromov \cite{MR864505} and Forstneri{\v{c}} \cite{MR880125}), 
we eliminate double points of the immersion, so as to 
obtain the main theorem (Theorem~\ref{thm:main}). 

Our argument focusing on the liftability gives back to 
an interesting corollary on stable maps (Corollary~\ref{cor:byprod}). 
Namely, we show that 
the singular set of a liftable stable map from a closed 
orientable $3$-manifold $M^3$ to the plane represents 
the trivial integral homology class in $H_1(M^3;\Z)$. 
This can be regarded as a refined version of the 
well-known Thom polynomial \cite{MR0087149} stating that 
the singular set of a stable map from a closed orientable 
$3$-manifold to the plane represents 
the trivial $\Z/2\Z$-coefficient homology class. 

In what follows, the term 
``$C^\infty$-smooth'' will be referred to simply as 
``smooth''. 
All manifolds and maps between manifolds shall be supposed 
to be smooth, unless otherwise stated. 

\section{Complex tangents} 
Let $f\co M^k\to\C^n$
be a smooth immersion. 
As mentioned in \S\ref{sect:intro}, 
a point $x\in M^k$ is said to be 
\textit{a complex tangent} if 
$df_x(T_xM^k)$ contains a complex line, that is, 
\[
df_x(T_xM^k)\cap J(df_x(T_xM^k))\ne\{0\}
\]
holds. If $f$ has no complex tangents 
it is said to be \textit{totally real}. 

We deal mainly with embeddings 
of closed orientable $3$-manifolds into $\C^3$. 
According to Lai~\cite[Theorem~2.3]{MR0314066}
(see also \cite[Proposition~(2.1)]{MR800003} and \cite[Proposition~3]{{MR2928578}}), 
for a smooth generic immersion of $M^3$ into $\C^3$, 
the set of complex tangents is empty or forms 
a codimension two submanifold of $M^3$. 

On the other hand, 
based on the $h$-principle due to Gromov \cite{MR864505}, 
it has been shown 
\cite{MR966952,MR880125} that any compact orientable $3$-manifold 
admits a totally real embedding in $\C^3$. 
More precisely, any immersion $f\co M^3\to\C^3$ 
of a compact orientable $3$-manifold $M^3$ is 
regularly homotopic to a totally real immersion, 
and moreover, 
if $f$ is regularly homotopic to an embedding 
then it is regularly homotopic to a totally real embedding. 
This implies, in a sense, that the existence 
of complex tangents is not 
an obstruction to totally reality. 
Therefore, our interests are rather in 
their global behaviours. 

Elgindi has initiated the study of topology of 
complex tangents of embeddings of the $3$-sphere 
in a series of papers \cite{MR2928578,MR3345506,Ali_pre}. 
In addition to the result mentioned in \S\ref{sect:intro}, 
he has shown in \cite{Ali_pre} that 
for any given knot $K$, an embedded circle in $S^3$, 
there exists a smooth embedding of 
$S^3$ into $\C^3$
with complex tangents forming a knot isotopic to $K$ 
\textit{or} 
a $2$-component link isotopic to two unlinked copies of $K$. 
The latter case,  however, cannot be excluded 
and after all it seems that 
the argument is facing a difficult trade-off. 
The problems of resolving this and 
dealing with knots and links in a general $3$-manifold 
are posed at the end of \cite{Ali_pre}. 
We will offer satisfactory solutions to these problems 
in \S\ref{sect:main}. 

\bigskip

Regarding the set of complex tangents, 
our first observation is the following. 

\begin{theorem}\label{thm:CT-homology}
For a smooth generic immersion of a 
closed orientable $3$-manifold $M^3$ into $\C^3$, 
the integral homology class represented by the set of complex tangents 
vanishes in $H_1(M^3;\Z)$. 
\end{theorem}

\begin{proof}
Denote by $G_{6,3}$ the Grassmann manifold 
of $3$-planes in $\R^6=\C^3$. 
Let $G_{6,3}^{\mathrm{TR}}$ be 
the open subset of $G_{6,3}$ consisting of 
totally real $3$-planes and 
put $W\coloneqq G_{6,3}\smallsetminus G_{6,3}^{\mathrm{TR}}$, 
which turns out to be a codimension two closed 
orientable submanifold 
of $G_{6,3}$ (see \cite[\S2 and Theorem~3]{MR3345506}). 

To a smooth immersion $f\co M^3\to\C^3$, 
we associate the Gauss map 
$\Gamma_f\co M^3\to G_{6,3}$ defined by 
$\Gamma_f(p)=df(T_pM^3)$. 
For a generic $f$, the Gauss map 
$\Gamma_f$ becomes a continuous map 
transverse to $W$ and 
the set of complex tangents of $f$ is 
just the codimension two closed orientable 
submanifold $\Gamma_f^{-1}(W)$ of $M^3$. 

As mentioned above, 
$f$ is regular homotopic to a 
totally real immersion \cite{MR966952,MR880125}. This implies 
that there exists a homotopy from 
$\Gamma_f$ to a map $M^3\to G_{6,3}$ 
with image inside $G_{6,3}^{\mathrm{TR}}$. 
Such a homotopy determines the map 
\[
\widetilde\Gamma_f\co M^3\times[0,1]\to G_{6,3}
\]
such that $\widetilde\Gamma_f$ restricted to $M^3\times\{0\}$ 
coincides with $\Gamma_f$ and 
$\widetilde\Gamma_f(M^3\times\{1\})\cap W=\emptyset$. 
By a small perturbation (fixed on $M^3\times\{0\}$)
if necessary, we can make 
$\widetilde\Gamma_f$ transverse to $W$. 
Then, the inverse image $\widetilde\Gamma_f^{-1}(W)$ 
gives an orientable submanifold bounded by 
$\Gamma_f^{-1}(W)$ in $M^3$, which implies that 
the set of complex tangents of $f$ is null-homologous in $M^3$. 
\end{proof}

\begin{remark}
The similar statement in the case 
of homology with coefficients in $\R$ 
has been proved in \cite{MR1360632,MR0314066}. 
\end{remark}

\section{Knots and links in $3$-manifolds}
The coherent band surgery for links in $\R^3$, which 
is equivalent to the move called \textit{nullification} or 
a sort of \textit{rational tangle surgery}, 
has been extensively studied 
particularly in relation with the study of 
enzyme actions on DNA (see \cite{MR1068451} for example). 
We consider here the notion of coherent band surgery 
for links in a general $3$-manifold $M^3$, 
that is, copies of circles embedded in $M^3$. 

\begin{figure}
\includegraphics[width=\textwidth]{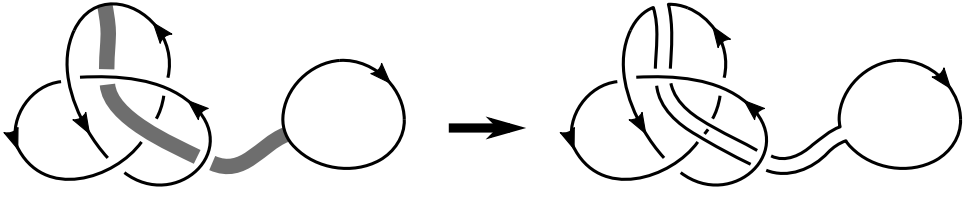}
\caption{The coherent band surgery}\label{fig:band}
\end{figure}

\begin{definition}
Let $L$ be a $1$-dimensional submanifold of 
$n$-manifold $M$ and $b\co I\times I\to M^n$
an embedding such that
$b(I\times I)\cap L=b(I\times\partial I)$, 
where $n\ge3$ and $I=[0,1]$. Then, 
\[
L'=(L\smallsetminus b(I\times\partial I))\cup b(\partial I\times I)
\]
is said to be the link obtained $L$ by
\textit{the band surgery} along the band $b$. 
If $L$ is an oriented link and $L'$ has 
the orientation compatible with 
$L\smallsetminus b(I\times\partial I)$, 
the link $L'$ is said to be obtained by
\textit{the coherent band surgery} (see Figure~\ref{fig:band}). 
\end{definition}

The following proposition has been 
implicitly proved in \cite[Lemma~3.9]{MR1359844}. 

\begin{figure}
\includegraphics[width=\textwidth]{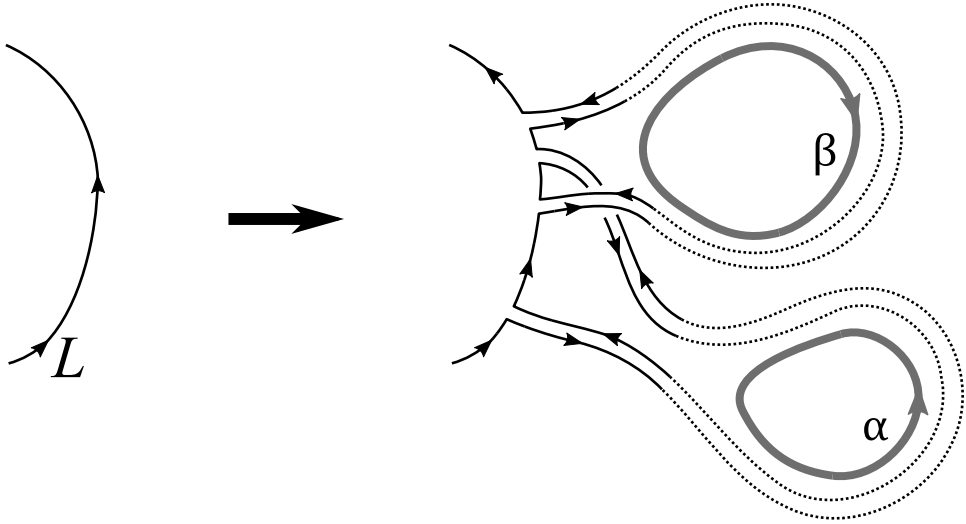}
\caption{The commutator through coherent band surgeries}\label{fig:commutator}
\end{figure}

\begin{figure}
\includegraphics[width=\textwidth]{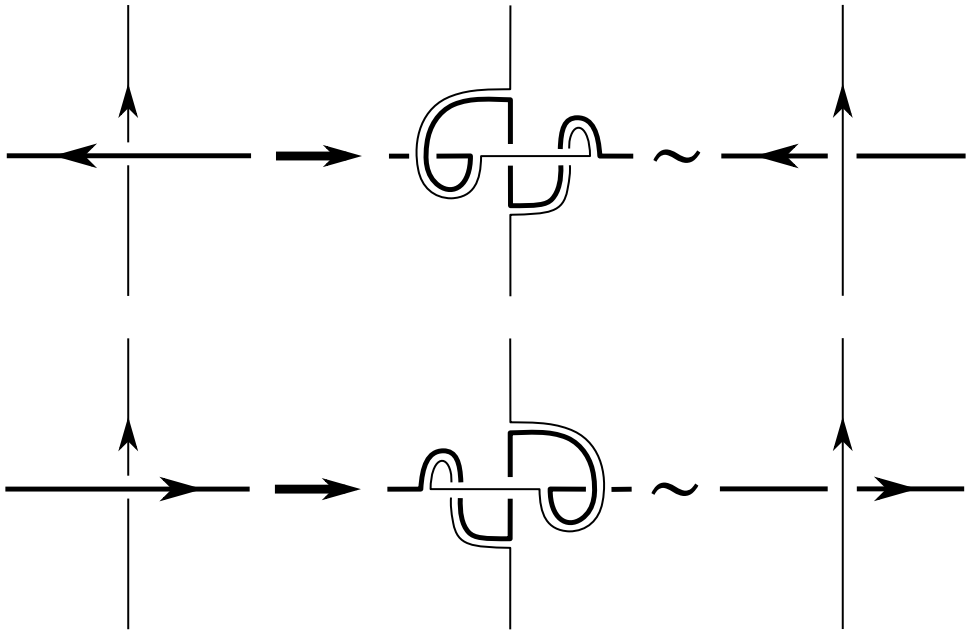}
\caption{The unknotting operation through coherent band surgeries}\label{fig:unknotting}
\end{figure}

\begin{proposition}\label{prop:coherent}
Let $L$ and $L'$ are closed oriented $1$-dimensional submanifolds 
of a closed $n$-dimensional manifold $M$ where $n\ge3$. Then, 
the homology class 
$[L]$ is equal to $[L']$ in $H_1(M;\Z)$ if and only if 
$L'$ is isotopic to a $1$-dimensional oriented submanifold 
obtained from $L$ by a finite iteration of 
coherent band surgeries. 
\end{proposition}

\begin{proof}
The proof is done essentially in the proof of \cite[Lemma~3.9]{MR1359844}. 

Since the coherent band surgery clearly does not change 
the homology class represented by the submanifold, 
the necessity is obvious. 

Now suppose that $L$ and $L'$ are integrally homologous in $M$. 
By suitable coherent band surgeries, we may 
assume they are both connected.  
Then the proof requires just the third and fourth paragraphs 
of the proof of \cite[Lemma~3.9]{MR1359844}. 
Namely, we first need to show that the action of 
a commutator $\alpha\beta\alpha^{-1}\beta^{-1}$, 
for $\alpha, \beta\in\pi_1(M^3,x)$ 
and a point $x\in L$, 
can be realised 
by an iteration of coherent band surgeries 
(see Figure~\ref{fig:commutator}). 
Thus, in view of the fact that 
$H_1(M;\Z)$ can be identified with 
the abelianisation of $\pi_1(M^3,x)$, 
we may assume that $L$ and $L'$ 
represents the same class in  $\pi_1(M^3,x)$; hence 
we then need to show that the unknotting operation 
(the ``crossing change'' up to isotopy) can be realised by an 
iteration of coherent band surgeries 
(see Figure~\ref{fig:unknotting}). 
See \cite[pages 1150--1151]{MR1359844} for details. 
\end{proof}

\section{Stable maps from $3$-manifolds to the plane}\label{subsect:stable}
We introduce a stable map
from $3$-manifolds to the plane. 
Although the notion of a stable map can be defined for 
more general source and target manifolds, 
those from $3$-manifolds to the plane, 
in particular, have been extensively studied 
(see \cite{MR790729,MR814689,MR1396772} and \cite{MR2205725} for example). 
Thus we adopt here a common definition in the dimensions.  
(see e.\,g.\, \cite{MR790729} or \cite[p.\,6]{MR814689}).

For a smooth map 
$f\co M\to N$ between smooth manifolds, 
we denote by $S(f)$ the set of singular points, that is, we put 
\[
S(f)=\left\{p\in M\thinspace\vrule\thinspace\mathrm{rank}\thinspace df_p\le\min(\dim{M},\dim{N})\right\}, 
\]
and call it \textit{the singular set} of $f$. 

\begin{definition}
A smooth map $f\co M^3\to\R^2$ is called 
\textit{a stable map} if it satisfies the following conditions. 
\begin{enumerate}
\item[I.]
Local conditions: \\
For each point $p\in M^3$, 
there exist local coordinates $(x,y,z)$ centred at $p$ and 
$(X,Y)$ centred at $f(p)$ with which 
$f$ has one of the following forms: 
\begin{enumerate}
\item[(L1) ]
$(X\circ f,Y\circ f)=(x,y)$\hfill ($p$ is called \textit{a regular point}),
\item[(L2) ]
$(X\circ f,Y\circ f)=(x,y^2+z^2)$\hfill ($p$ is called \textit{a definite fold point}),
\item[(L3) ]
$(X\circ f,Y\circ f)=(x,y^2-z^2)$\hfill ($p$ is called \textit{a indefinite fold point}),
\item[(L4) ]
$(X\circ f,Y\circ f)=(x,xy+y^3+z^2)$\hfill ($p$ is called \textit{a cusp point}). 
\end{enumerate}
\item[II.]
Global conditions: 
\begin{enumerate}
\item[(G1) ]
for each cusp point $p\in M^3$, 
we have $f^{-1}\left(f(p)\right)\cap S(f)=\{p\}$.
\item[(G2) ]
$f$ restricted to $(S(f)\smallsetminus\{\text{cusp points}\})$ 
is an immersion with normal crossings, 
\end{enumerate}
\end{enumerate}
\end{definition}

\begin{remark}
The set of all stable maps $M^3\to\R^2$
from a compact $3$-manifold $M^3$ 
to the plane is open and dense 
in the mapping space 
$C^\infty(M^3,\R^2)$, the set of 
all smooth map $M^3\to\R^2$ 
endowed with the Whitney $C^\infty$-topology \cite{MR0293670}. 
Hence, every smooth map $M^3\to\R^2$ 
can be approximated 
by a stable map. 
\end{remark}

\begin{figure}
\includegraphics[width=\textwidth]{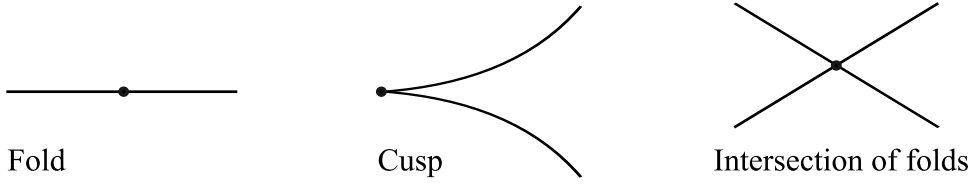}
\caption{Local images of singular points}\label{fig:stable}
\end{figure}

\begin{remark}
For a stable map $f\co M^3\to\R^2$ 
from a compact $3$-manifold $M^3$ 
to the plane, 
its singular set $S(f)$ forms 
a compact smooth $1$-dimensional 
submanifold of $M^3$. 
It consists of 
smooth arcs of definite folds or indefinite folds, and 
isolated cusp points, where definite and  
indefinite fold arcs meet. 
Figure~\ref{fig:stable} depicts 
the local image of the singular points. 

The $2$-dimensional regions divided 
by the lines consists of 
the images of regular points. 
The regular fibre, the inverse image of a regular point, 
consists of several copies of circles. 
Therefore, if we describe 
such regular fibres 
and how they degenerate in 
crossing the singular lines, 
we can recover the given map $f$ locally. 
This is successfully done in the fundamental paper \cite{MR2205725} 
due to Minoru Yamamoto, 
who has studied in detail stable maps from $3$-manifolds to the plane 
and their deformations with lots of clear figures. 
We will often refer 
the figures of his paper \cite{MR2205725}. 
\end{remark}

A generic homotopy $f_t\co M^3\to\R^2\ (t\in[-1,1])$
between two stable maps $f_0$ and $f_1$ has 
has been studied in 
\cite{chincaro,sotomayor}. 
The germ of such a generic homotopy $f_t$, 
in suitable local coordinates $(x,y,z)$ of 
$M^3$ and $(X,Y)$ of $\R^2$, 
is given by one of the following: 
\begin{enumerate}
\item[(i)]
$(X\circ f_t,Y\circ f_t)=(x,y^3+yx^2+z^2+yt)$\hfill (\textit{Lips}),
\item[(ii)]
$(X\circ f_t,Y\circ f_t)=(x,y^3-yx^2+z^2+yt)$\hfill (\textit{Beaks}),
\item[(iii)]
$(X\circ f_t,Y\circ f_t)=(x,y^4+yx\pm z^2+y^2t)$\hfill (\textit{Swallowtail}); 
\end{enumerate}
in addition, 
\begin{enumerate}
\item[(iv)]
an \textit{intersection of a fold and a cusp}, 
\item[(v)]
a \textit{non-transversal intersection of two folds}, 
\item[(vi)]
an \textit{intersection of three folds} 
\end{enumerate}
may occur 
as codimension one multigerms. 
(see Figure~\ref{fig:homotopy}). 
Each of the above homotopies passes through a non-stable map 
at $t=0$, \textit{a bifurcation point}. 
Note that only in (i), (ii) and (iii) 
the types of singularities 
changes through the bifurcation point. 
Note further that the case (iii) 
does not change \textit{the set} of singular points. 

\begin{figure}
\includegraphics[width=\textwidth]{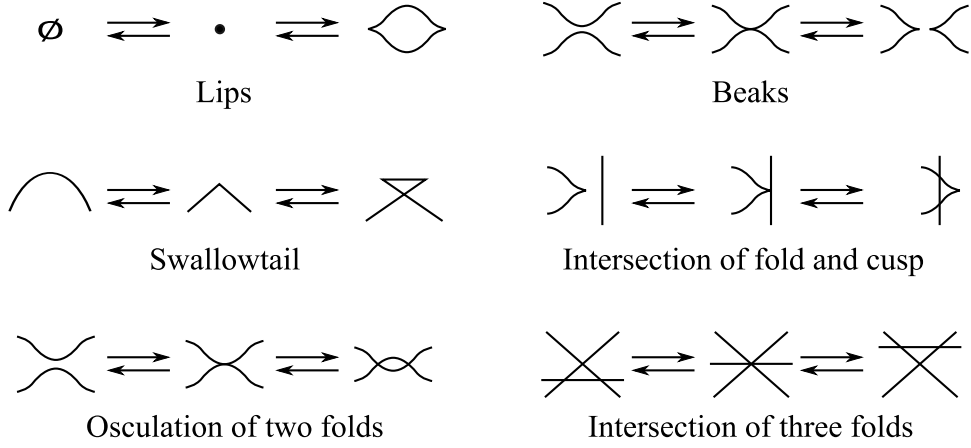}
\caption{Generic homotopies}\label{fig:homotopy}
\end{figure}

\section{Saeki's theorem}\label{subsect:saeki}
A striking result that 
every knot or link in $S^3$ is 
the singular set of a stable map from $S^3$ to the plane 
is a consequence of the following 
theorem due to Saeki \cite[Theorem~2.2]{MR1359844}. 

\begin{theorem}[Saeki {\cite[Corollary~6.3]{MR1359844}}]\label{thm:saeki}
Let $M^3$ be a closed orientable $3$-manifold and $L$ 
be a closed $1$-dimensional submanifold of $M^3$. Then
there exists a stable map $f\co M^3\to\R^2$ with $S(f)=L$
if and only if the $\Z/2\Z$-coefficient homology class $[L]_2$ 
vanishes in $H_1(M^3;\Z/2\Z)$. 
\end{theorem}

We review the outline of Saeki's proof. 

His first observation is that 
any two $1$-dimensional submanifolds (links) 
$L_1$ and $L_2$ in $M^3$ can be 
connected by a finite iteration of (not necessarily coherent) 
band surgeries 
if and only if $[L_1]_2=[L_2]_2$ in $H_1(M^3;\Z/2\Z)$ 
(\cite[Lemma~3.9]{MR1359844}, compare it with 
Proposition~\ref{prop:coherent}). 

The second point is that the Beaks, one of the seven generic 
homotopies introduced in \S\ref{subsect:stable}, 
affect the singular set just as the band surgery. 
More precisely, we can deform a given stable map 
$f\co M^3\to\R^2$, 
so that an arbitrary band surgery 
is performed on the singular set $S(f)$, 
by suitably iterating the following 
four types of generic homotopies (see \cite[Remark~4.2]{MR1359844}). 

\begin{enumerate}
\item[(H0)]
\textit{Lips from left to right} in Figure~\ref{fig:homotopy}, 

\item[(H1)]
\textit{Beaks from right to left} in Figure~\ref{fig:homotopy}, 

\item[(H2)]
\textit{Swallowtail from left to right} in Figure~\ref{fig:homotopy}, and 

\item[(H3)]
\textit{Intersection of a fold and a cusp from left to right} in Figure~\ref{fig:homotopy}.
\end{enumerate}
Specifically, these homotopies are used as follows. 
First, if the stable map has the empty singular set, 
by using (H0) we can make it nonempty 
(note that we do not need 
this since a stable map to the plane 
necessarily has the nonempty singular set). 
Then, if we need to perform a band surgery 
in some place on the singular set, 
we use (H2) to yield cusps in the correct locations, 
which serve as ``footholds'' for the necessary ``band''; 
and then extend the band from one foot 
toward the other by a series of (H3). 
The Beaks (H1) 
in the final stretch complete 
the desired band surgery on the singular set.  

Thus, recalling that, among the above four 
generic homotopies, 
only the Lips and the Beaks 
may possibly change the isotopy class 
(link type) of $S(f)\subset M^3$ (and the Lips are 
not used in our case), the proof 
is outlined as follows. 
\textit{%
The ``only if'' part is nothing but the Thom 
polynomial: for any stable map 
$f\co M^3\to\R^2$ the homology class $[S(f)]_2$ 
represents the dual of the second Stiefel--Whitney class 
$w_2(M^3)$ and hence always vanishes in 
$H_1(M^3;\Z/2\Z)$ 
(Thom \cite{MR0087149}). 
For the ``if'' part, 
beginning with an arbitrary stable map 
$f_{\mathit{init.}}\co M^3\to\R^2$ we can modify it 
into $f'\co M^3\to\R^2$ 
through a finite iteration of (H1), (H2) and (H3) 
so that the singular set $S(f')$ is isotopic to any given $L$, 
as long as $[L]_2=0\in H_1(M^3;\Z/2\Z)$. 
Finally, by composing a suitable homeomorphism 
(derived from an ambient isotopy) of $M^3$ with $f'$ 
we obtain the desired map $f\co M^3\to\R^2$ 
with $S(f)=L$ (see \cite[p.\,1155]{MR1359844}).} 

\begin{remark}\label{rmk:init}
From now on, we will call the above stable map 
$f_{\mathit{init.}}\co M^3\to\R^2$ 
\textit{the initial stable map} 
in Saeki's construction. 
\end{remark}

\section{Liftable stable maps and complex tangents}\label{subsect:lift}
Given a stable map 
the problem whether it can be lifted 
to an immersion (or an embedding) 
has been studied by many authors 
(the references in \cite{MR3153918} might be convenient), 
including Haefliger \cite{MR0116357} and 
Harold Levine \cite{MR814689}. 
Haefliger \cite{MR0116357} has studied the lifting problem 
of a stable map from a surface to $\R^2$ and 
obtained a necessary and sufficient condition 
for being lifted to an immersion into $\R^3$. 
H.~Levine \cite{MR1101845} has studied 
the analogous problem of the existence (and classification \cite{MR814689}) 
of immersions into $\R^4$ over stable maps 
from $3$-manifolds to $\R^2$. 

\begin{definition}
Let $f\co M^3\to\R^2$ be a stable map 
from an orientable $3$-manifold to $\R^2$. 
We shall say that $f$ is \textit{liftable} or 
has \textit{an immersion lift $\widetilde{f}$} in $\R^4$ if 
there exists an immersion 
$\widetilde{f}\co M^3\to\R^4$ such that 
$\pi\circ\widetilde{f}=f$ for the projection 
$\pi\co\R^4\to\R^2$. 
\end{definition}

\begin{remark}\label{rmk:lift}
Note that any 
closed orientable $3$-manifold $M^3$ admits a 
liftable stable map to $\R^2$.  
To obtain such a stable map, 
we only need to compose 
any generic immersion 
$M^3\to\R^4$ with 
a generic projection $\R^4\to\R^2$ \cite{MR0362393}.  
\end{remark}

The following seems the	 first indication 
of a relation between liftable stable maps and complex tangents. 

\begin{theorem}\label{thm:imm}
Let $g_1=(f_1,f_2)\co M^3\to\R^2$ be a stable map 
of a closed orientable $3$-manifold 
which has an immersion lift 
$\widetilde{g_1}=(f_1,f_2,f_3,f_4)$. Then, 
the map 
\[
G=(f_1,f_2,f_3,f_4,f_1,-f_2)\co M^3\longrightarrow\R^6=\C^3
\]
defines a smooth immersion of $M^3$ into $\C^3$ 
the set of whose complex tangents coincides 
with the singular set $S(g_1)$ of $g_1$. 
\end{theorem}

\begin{proof}
Put $g_2:=(f_3,f_4)$ and $g_3:=(f_1, -f_2)$, so that 
we have 
\[
G=(g_1,g_2,g_3)\co M^3\longrightarrow\C^3. 
\]

First, we show that the immersion $G$ is totally 
real on $M^3\smallsetminus S(g_1)$. 
Take any regular point $p$ of $g_1$ and 
consider the differential map 
\[
dG_p: T_pM^3\longrightarrow T_{G(p)}\C^3=\C^3=\C_1\oplus\C_2\oplus\C_3
\] 
at $p$. 
Then, the image $dG_p(T_pM^3)$ contains no 
complex line by the following reason. 

Suppose that $dG_p(T_pM^3)$ contains a complex line $l_p$. 
Then, via the projections $T_{G(p)}\C^3\to\C_i$ for $i=1,2$ and $3$, 
the line $l_p$ should be mapped holomorphically onto or zero to each $\C_i$. 
Since $\mathrm{dim}_{\R}\ker {(dg_1)_p}=1$, we see that 
$l_p$ should be mapped holomorphically onto both $\C_1$ and $\C_3$. 
But this is impossible since, 
by the definition of $G$, if $l_p$ is mapped holomorphically onto one 
then it should be mapped anti-holomorphically onto the other. 
Thus, $G$ is totally real on $M^3\smallsetminus S(g_1)$. 

Next, we show that for any singular point $p\in S(g_1)$, 
the image 
$dG_p(T_pM^3)$ contains a complex line. 
Since $g_1$ is a stable map, 
along the smooth link $L\coloneqq S(g_1)$, 
$\ker{dg_1\restr_{L}}$ is a real $2$-dimensional vector bundle. 
Furthermore, since $\widetilde{g_1}=(g_1, g_2)$ is an immersion, 
$(dg_2)_p$, for $p\in L$, gives a linear isomorphism between $\ker{(dg_1)}_p$ and $\C_2$.  
Since 
$\ker{(dg_1)_p}=\ker{(dg_3)_p}$, 
the image $dG_p(\ker{(dg_1)}_p)$ is equal to $(dg_2)_p(\ker{(dg_1)}_p)$, 
which is nothing but the complex line $T_{g_2(p)}\C=\C_2$ in $T_{G(p)}\C^3$. 

We have thus shown that the set of complex tangents of 
the immersion $G$ coincides with the 
singular set $S(g_1)$ of the stable map $g_1$. 
\end{proof}

\begin{remark}~\label{rmk:trans}
In Theorem~\ref{thm:imm}, 
since the singular set $L=S(g_1)$ is a $1$-dimensional 
submanifold of $M^3$ 
and the map $g_1=(f_1,f_2)$ restricted to $L$, except 
at isolated cusp points, is 
a self-transverse immersion to the plane, 
we may assume that the immersion 
$\widetilde{g_1}=(f_1,f_2,f_3,f_4)$ restricted to 
$L$ is an embedding, 
by slightly perturbing $f_3$ and $f_4$ if necessary. 
Then, there is a tubular neighbourhood $N$ of 
$L$ such that $\widetilde{g_1}\restr_N$ is an embedding. 
Thus we can choose the immersion $G$ so that its restriction 
to $N$ is an embedding. 
\end{remark}

The following is an easy consequence of 
Theorems~\ref{thm:CT-homology} and \ref{thm:imm}. 
As mentioned in \S\ref{sect:intro}, it is intriguing to compare Corollary~\ref{cor:byprod} 
with the usual Thom polynomial \cite{MR0087149}. 

\begin{corollary}\label{cor:byprod}
Let $f\co M^3\to\R^2$ be a liftable stable map 
from a closed orientable $3$-manifold $M^3$. 
Then, $[S(f)]=0\in H_1(M;\Z)$. 
\end{corollary}

\section{An orientation of the singular set of liftable stable maps and an extension of Saeki's theorem}\label{subsect:ext}
As mentioned in \S\ref{subsect:lift}, 
H.~Levine \cite{MR1101845} has studied 
the lifting problem of stable maps 
from $3$-manifolds to the plane and 
given an example of a non-liftable stable map 
from an orientable $3$-manifold 
(\cite[Example~2, p.\,288]{MR1101845}).
His example is based on 
a certain necessary condition for the existence 
of an immersion lift in terms of 
the rotation numbers of regular fibres and an 
appropriate orientation on the singular set 
(see also 
\cite[Theorem, p.\,55 and Proposition, p.\,59]{MR814689}). 
We recall it briefly here. 

\begin{figure}
\includegraphics[width=\textwidth]{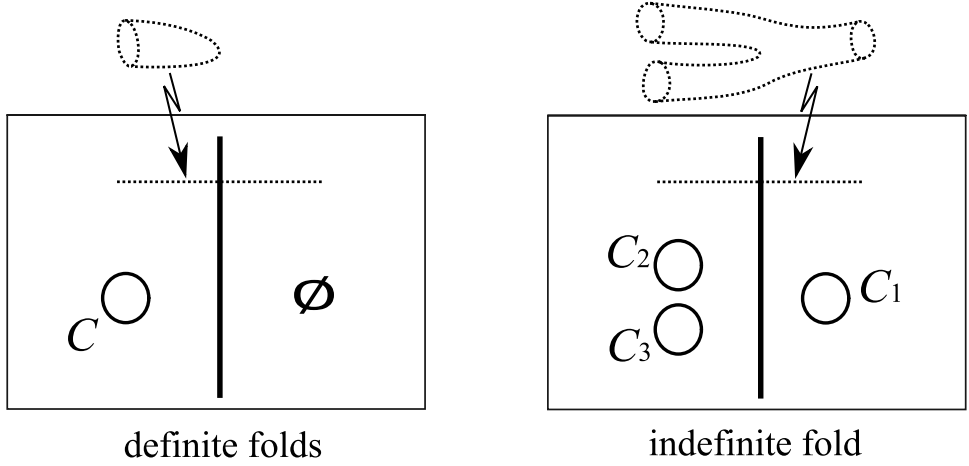}
\caption{A good orientation of the singular set}\label{fig:orientation}
\end{figure}

Let $f\co M^3\to\R^2$ be a stable map 
from an oriented $3$-manifold with 
an immersion lift 
$\widetilde{f}=(f,h)\co M^3\to\R^4=\R^2\times\R^2$, 
and $\pi\co\R^4\to\R^2$ the projection onto the first factor. 
Then, each component $C$ of the regular fibre 
of $f$ over a point $x\in\R^2$, 
which can be compatibly oriented 
with respect to the orientations of $M^3$ and $\R^2$, 
is immersed by $h$ into $\R^2=\pi^{-1}(x)$, so that we can consider 
the rotation number $r(C)$ of this immersion. 
Then, according to \cite[p.\,288]{MR1101845}, 
for such a liftable stable map $f$ 
we can take an orientation of 
the singular set $S(f)$ such that the 
following equations hold in Figure~\ref{fig:orientation}, 
which depicts the image of definite and indefinite fold points 
up to regular circle components: 
\begin{itemize}
\item 
$r(C)=1$ (resp.\ $=-1$) if the arc of definite folds is oriented upward (resp.\ downward),  and 
\item
$r(C_2)+r(C_3)-r(C_1)=1$ (resp.\ $=-1$) if the arc of indefinite folds is oriented upward (resp.\ downward). 
\end{itemize}
We shall call it \textit{a good orientation} of $S(f)$ 
in what follows. 
Note that the choice of a good orientation is not unique. 

\bigskip
The rest of this section is devoted to 
extending Saeki's theorem slightly. 
Namely, we show that 
in Theorem~\ref{thm:saeki}
the stable map $f$ can be chosen to be liftable,
if the integral homology class $[L]$ represented by 
the given link $L$ vanishes. 

\begin{theorem}\label{thm:ext}
Let $M^3$ be a closed orientable $3$-manifold and $L$ 
be a closed $1$-dimensional submanifold of $M^3$. Then, 
there exists a liftable stable map $f\co M^3\to\R^2$ 
with $S(f)=L$ if and only if 
the $\Z$-coefficient homology class $[L]$ vanishes 
in $H_1(M^3;\Z)$. 
\end{theorem}

\begin{figure}
\includegraphics[width=\textwidth]{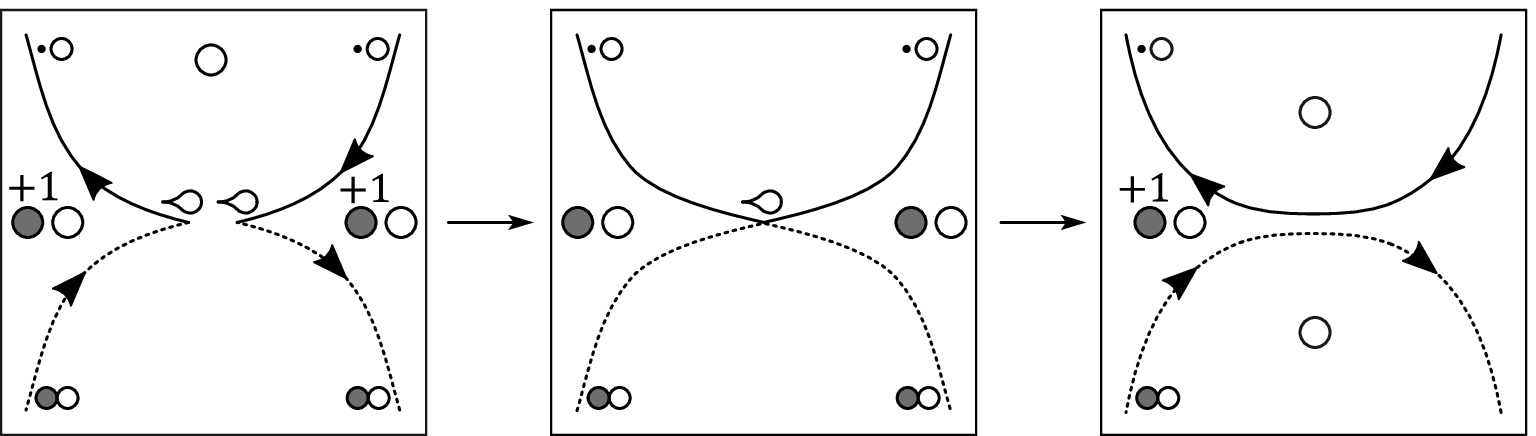}
\includegraphics[width=\textwidth]{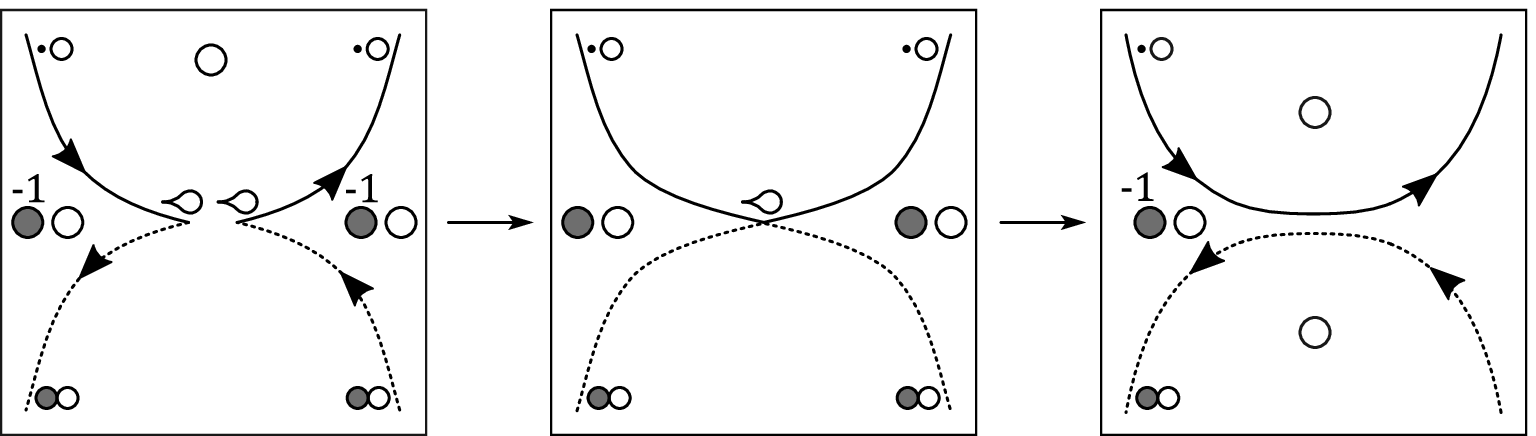}
\caption{Beaks with coherent orientations (see \cite[Figure~6(a) (2) and Figure~8(a) $\mathrm{III}^a(b)$]{MR2205725})}\label{fig:s1}
\end{figure}

\begin{figure}
\includegraphics[width=\textwidth]{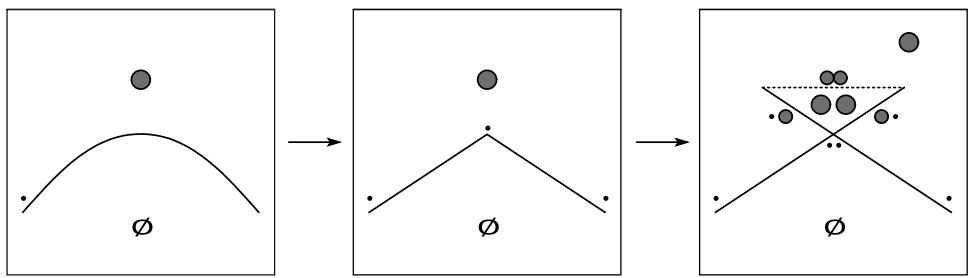}
\caption{D-swallowtail(see \cite[Figure~6(a) (3) and Figure~8(a) $\mathrm{III}^b$]{MR2205725})}\label{fig:s2}
\end{figure}

\begin{figure}
\includegraphics[width=\textwidth]{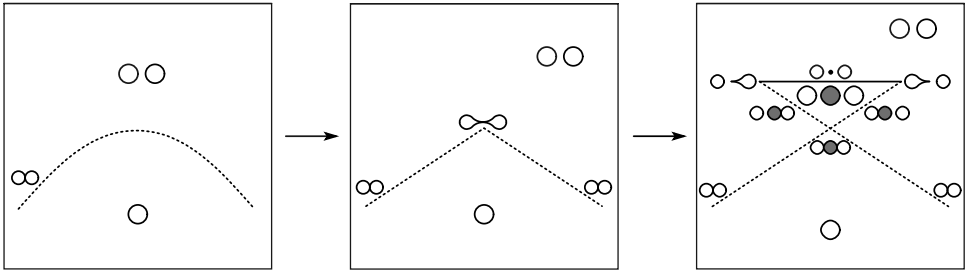}
\caption{I-swallowtail (see \cite[Figure~6(a) (4) and Figure~8(a) $\mathrm{III}^c$]{MR2205725})}\label{fig:s3}
\end{figure}

\begin{figure}
\includegraphics[width=\textwidth]{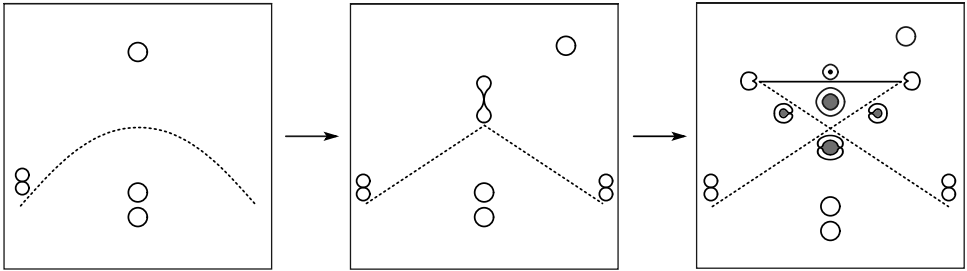}
\caption{I-swallowtail (see \cite[Figure~6(a) (4) and Figure~8(a) $\mathrm{III}^d$]{MR2205725})}\label{fig:s4}
\end{figure}

\begin{figure}
\includegraphics[width=\textwidth]{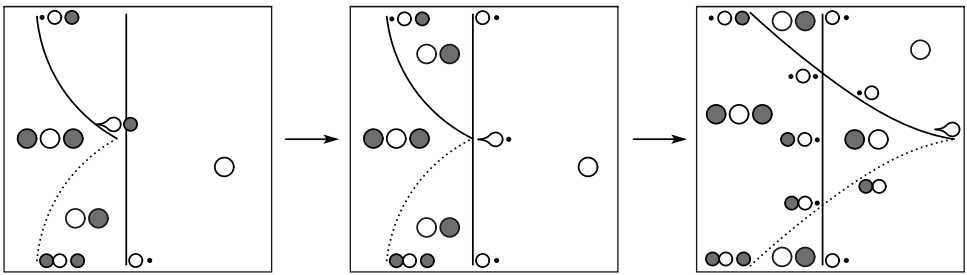}
\caption{cusp-plus-D fold (type 1) (see \cite[Figure~6(b) (5) and Figure~8(a) $\mathrm{III}_1^{0,a}$]{MR2205725})}\label{fig:s5}
\end{figure}

\begin{figure}
\includegraphics[width=\textwidth]{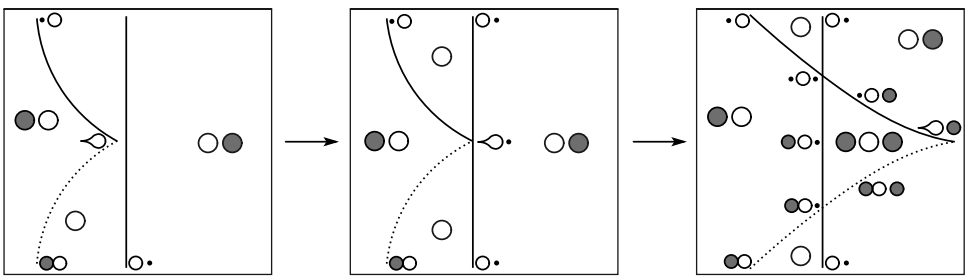}
\caption{cusp-plus-D fold (type 2) (see \cite[Figure~6(b) (6) and Figure~8(a) $\mathrm{III}_2^{0,a}$]{MR2205725})}\label{fig:s6}
\end{figure}

\begin{figure}
\includegraphics[width=\textwidth]{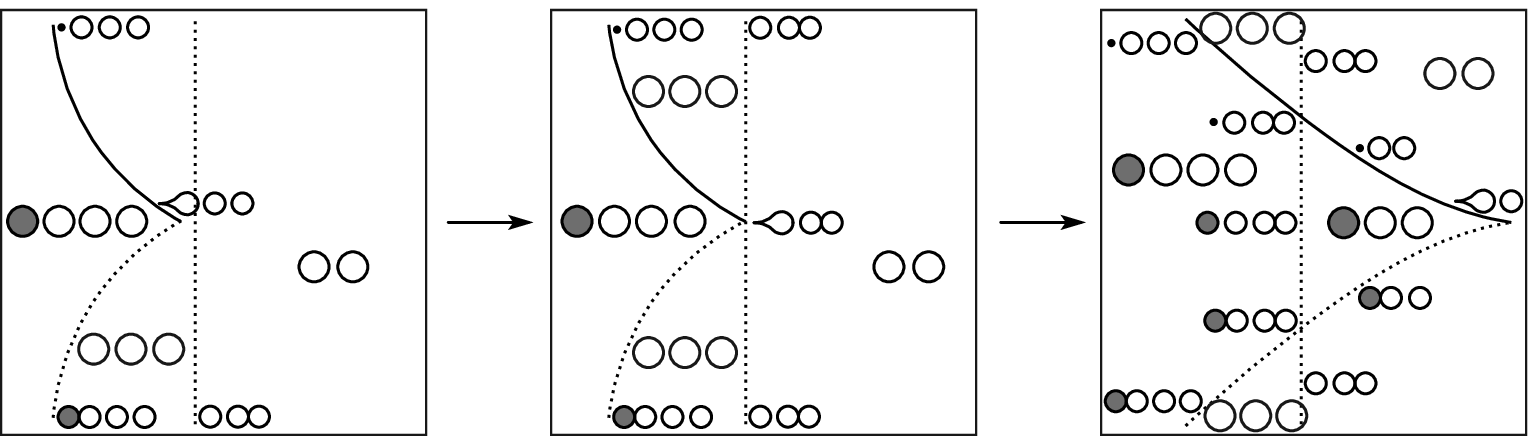}
\caption{cusp-plus-I fold (see \cite[Figure~6(b) (7) and Figure~8(a) $\mathrm{III}_1^{1,a}$]{MR2205725})}\label{fig:s7}
\end{figure}

\begin{figure}
\includegraphics[width=\textwidth]{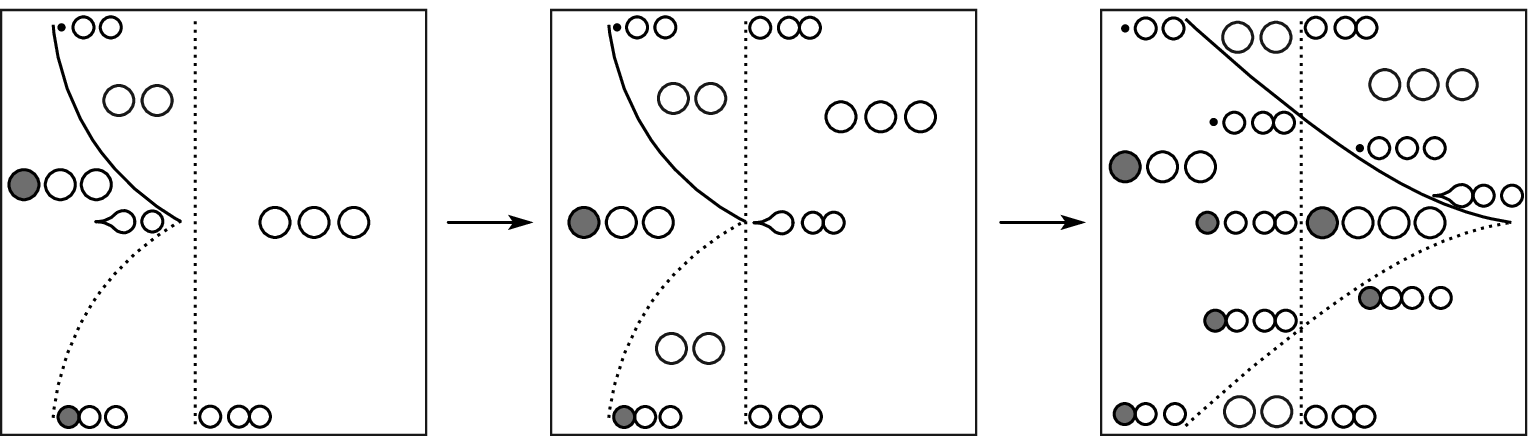}
\caption{cusp-plus-I fold (see \cite[Figure~6(b) (7) Figure~8(a) $\mathrm{III}_2^{1,a}$]{MR2205725}}\label{fig:s8}
\end{figure}

\begin{figure}
\includegraphics[width=\textwidth]{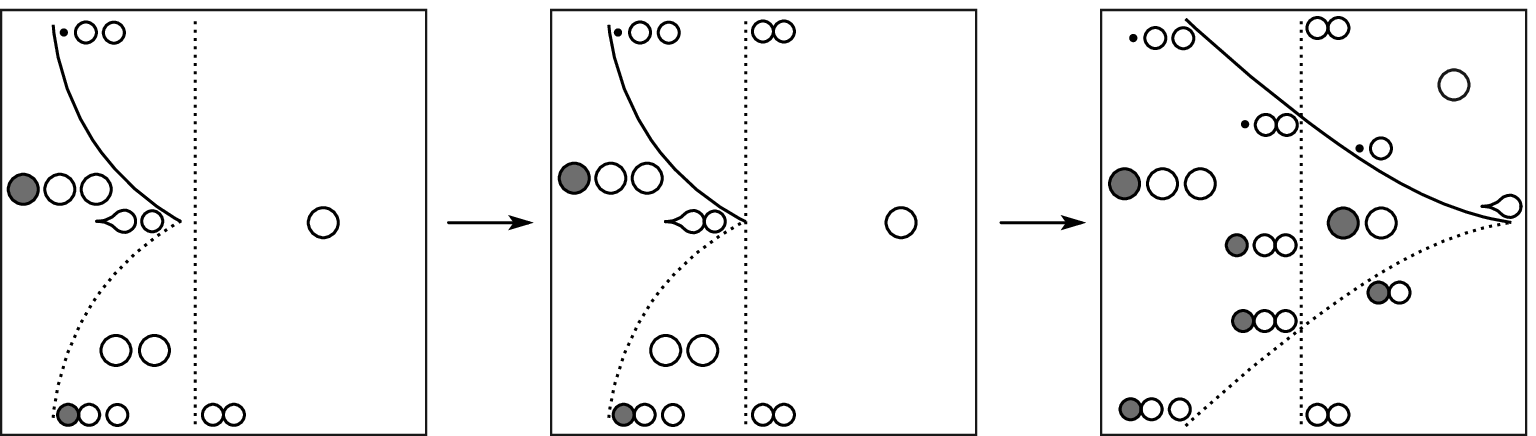}
\caption{cusp-plus-I fold (see \cite[Figure~6(b) (7) Figure~8(a) $\mathrm{III}_1^e$]{MR2205725}}\label{fig:s9}
\end{figure}

\begin{figure}
\includegraphics[width=\textwidth]{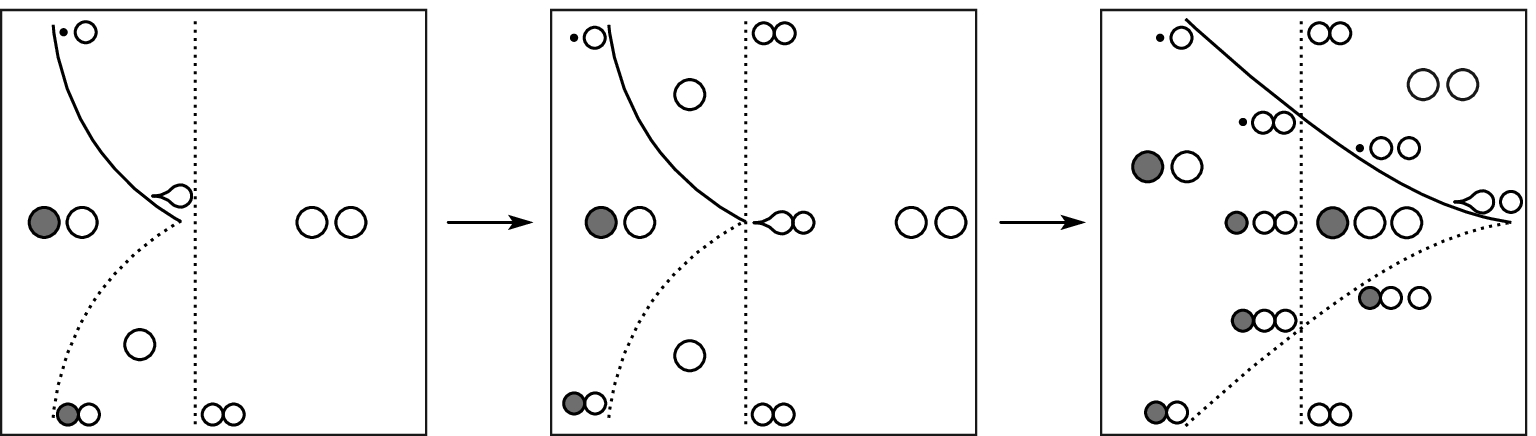}
\caption{cusp-plus-I fold (see \cite[Figure~6(b) (7) Figure~8(a) $\mathrm{III}_2^e$]{MR2205725}}\label{fig:s10}
\end{figure}

\begin{proof}
The ``only if'' part is nothing but 
Corollary~\ref{cor:byprod}. 

Suppose that $[L]=0\in H_1(M^3;\Z)$ and 
consider any orientation on $L$. 
In Saeki's construction, 
we can choose the initial stable map $f_{\mathit{init}}\co M^3\to\R^2$ 
(Remark~\ref{rmk:init}) 
so that it has an immersion lift in $\R^4$ (Remark~\ref{rmk:lift}). 
Take a good orientation on the singular set $S(f_{\mathit{init}})$ 
(see the beginning of this section). 
Then, 
$[S(f_{\mathit{init}})]=0\in H_1(M^3;\Z)$ 
by Corollary~\ref{cor:byprod}, and hence 
$S(f_{\mathit{init}})$ can be related to $L$ 
through a finite iteration of cohrent band surgeries 
by Proposition~\ref{prop:coherent}. 
Therefore, 
we only have to check that 
the coherent cases of (H1) in addition to (H2) and (H3) 
in \S\ref{subsect:saeki} 
keep the ``liftability''. 

To see this, 
we first list the all necessary local homotopies, 
including the information how 
regular fibres degenerate in crossing 
the image of the singular points and 
how the singular fibres 
are deformed during the homotopies 
(up to regular circle components). 
This is carried out in Figures~\ref{fig:s1}--\ref{fig:s10}, 
based on the classification given in \cite[Theorem~4.7]{MR2205725}. 
Our Figures~\ref{fig:s1}--\ref{fig:s10} just correspond to 
the eleven figures appearing in \cite[Figure~8(a)]{MR2205725} 
(the case $\mathrm{III}^a(l)$ corresponding to Lips 
is excluded). 

In Figures~\ref{fig:s1}--\ref{fig:s10}, we use 
the solid lines for definite folds 
and the dotted line for indefinite folds. 
Bigger circles represent the regular fibre 
over each point of the region and 
smaller circles drawn near the lines 
indicate how they are deformed and degenerate there. 
An explanation about the circles with shade is given below. 

What we want to show is that in 
each homotopy in Figures~\ref{fig:s1}--\ref{fig:s10} 
if we have an immersion lift in $\R^4$ at the 
beginning (the left) 
then we can construct an immersion lift in $\R^4$ 
at the end (the right) consistently. 
Actually we will prove more, that is, 
each generic homotopy above can be covered by 
a regular homotopy in $\R^4$, 
if it has an immersion lift in $\R^4$ 
at the beginning. 

Now suppose that we have an immersion lift in $\R^4$, 
in the left column of each figure. 
Then, the regular fibre over each point of 
the $2$-dimensional regions 
--- the disjoint union of copies of circles 
--- is immersed into $\R^2$ (that is 
a fibre of the projection $\R^4\to\R^2$). 
At this point we see that 
some fibre components, vanishing 
as they travel towards the definite folds, 
should be immersed into $\R^2$ 
with rotation number $\pm1$. 
In Figures~\ref{fig:s1}--\ref{fig:s10}, we shade such circles. 
Particularly in Figure~\ref{fig:s1}, 
the goodness of the orientation on $S(f)$ 
determines whether $+1$ or $-1$ should be chosen; 
the numbers in Figure~\ref{fig:s1} refer it (but what is 
important here is that 
the two shaded circles of the left column 
have the same number in each of the two cases 
of Figure~\ref{fig:s1}). 
Note that we do not know how 
other circles (with no shade) 
are immersed into $\R^2$. 

Thus, 
by focusing on those shaded circles, 
it turns out that we can construct an immersion lift 
for the right column, 
so that 
circles with shade are immersed trivially 
(with rotation number $\pm1$) 
and the other circles are immsersed 
in the inherited ways from the left column. 

For example, in the right column of Figure~\ref{fig:s5}, 
we need to determine how the two circles of 
the (newly generated) central region are immersed into $\R^2$; 
the figure shows that we can do this by immersing 
the shaded circle trivially and 
the other circle similarly to 
the circle of the rightmost region. 
The other figures can be likewise understood. 
However, Figure~\ref{fig:s4} seems a little complicated. 
In the right column of Figure~\ref{fig:s4}, it is obvious that 
the two circles in the bottom region should be 
immersed into $\R^2$ in the similar way as those of the left column. 
The shade of the singular fibre at the intersection of the lines 
indicates that these immersed circles, 
nearing the intersection of the singular lines, 
osculate at two points in such a way that 
we could span an immersed disk at the shaded potion. 
For the two circles of the central region, 
one should be immersed similarly as the circle of the top region 
and the other (with shade) should be immersed trivially. 
We can thus determine a consistent immersion lift 
for the right column. 
(In Figure~\ref{fig:s4}, all the shaded portions in the 
right column shrink to a point at the bifurcation point, 
that is, in the center column). 

Finally we see that the situations of the 
degenerations of fibres in the left and the right columns 
can be continuously connected via
the center column (the bifurcation point). 
Therefore, we could obtain 
a covering regular homotopy in $\R^4$
of each generic homotopy. 
\end{proof}

\begin{remark}\label{rmk:rhom}
For an immersion of $M^3$ in $\R^4$ 
the precomposition by a homeomorphism of $M^3$ does not 
change the regular homotopy class of the immersion 
(see \cite{MR0172302} and \cite[Remark~3.4]{MR2303517}). 
Therefore, 
in view of the last paragraph of the proof of Theorem~\ref{thm:ext}, 
the immersion lift of the resultant stable map and 
that of the initial stable map belong
to the same regular homotopy class.  
\end{remark}

\section{Knots and links of complex tangents}\label{sect:main}
We are now ready to state the main theorem. 

\begin{theorem}\label{thm:main}
Let $M^3$ be a closed orientable $3$-manifold and $L$ 
be a closed $1$-dimensional submanifold of $M^3$. 
Then, there exists a smooth embedding $F\co M^3\to\C^3$ 
the set of whose complex tangents coincides with $L$ 
if and only if $[L]=0\in H_1(M^3;\Z)$. 
\end{theorem}

\begin{proof}
The ``only if'' part follows from 
Theorem~\ref{thm:CT-homology}. 
Suppose that $[L]=0\in H_1(M^3;\Z)$. 

By Whitney's theorem, 
we can 
embed $M^3$ into $\R^6$ such that it has 
the trivial normal bundle. 
Furthermore, by the Compression Theorem \cite{MR1833749}, 
such an embedding $G'\co M^3\to \R^6$ 
can be chosen so that 
its composition with the projection to $\R^4$ 
becomes an immersion, which we denote 
by $\widetilde{f'}\co M^3\to \R^4$. 
By further composing $\widetilde{f'}$ with a generic projection $\R^4\to \R^2$, 
we obtain a stable map, denoted by $f'=(f'_1, f'_2)\co M^3\to \R^2$. 

We have thus obtained the liftable stable map 
$f'=(f'_1, f'_2)\co  M^3\to\R^2$ with the immersion lift 
$\widetilde{f'}=(f'_1,f'_2,f'_3,f'_4)\co M^3\to\R^4$, whose 
composition $j\circ \widetilde{f'}$ with the inclusion 
$j\co \R^4\to\R^6$ is regularly homotopic to 
the embedding $G'\co M^3\to\R^6$. 

Since $[L]=0$, 
by using this stable map $f'$ as the initial stable map 
(see Remark~\ref{rmk:init}) in the proof of 
Theorem~\ref{thm:ext}, 
we obtain a stable map $f$ which satisfies $S(f)=L$ 
and has an immersion lift 
$\widetilde{f}=(f_1,f_2,f_3,f_4)\co M^3\to\R^4$. 
By Remark~\ref{rmk:rhom}, furthermore, 
we see that $\widetilde{f}$ is 
regularly homotopic to $\widetilde{f'}$. 

By Theorem~\ref{thm:imm}, we obtain 
the immersion $G''=(f_1, f_2, f_3, f_4, f_1, -f_2)$ 
whose complex tangents forms $S(f)=L$. 
Then, it is clear that $G''$ 
is regularly homotopic to 
the embedding $G'\co M^3\to \R^6$. 
By Remark~\ref{rmk:trans}, we may assume that 
the immersion $G''$ is already an embedding on 
a tubular neighbourhood $N$ of $L$. 
Furthermore, 
since the condition of totally reality is an open condition, 
by slightly perturbing $G''$ on $M^3\smallsetminus N$, 
we obtain 
a new immersion $G$ which 
has only transverse double points away from $N$ and 
is still totally real on $M^3\smallsetminus N$. 
Thus, the self-transverse immersion $G$ has only 
isolated double points lying in $M^3\smallsetminus N$ and 
the set of complex tangents of $G$ coincides with $L$. 

Finally we apply to $G\restr_{M^3\smallsetminus N}$ 
the relative $h$-principle for 
totally real embeddings 
(see Gromov \cite{MR864505}, Eliashberg and Mishachev \cite{MR1909245}, 
and Forstneri{\v{c}} \cite{MR880125} for example). 
Since the immersion $G$ 
is regularly homotopic to an embedding, 
is totally real on $M^3\smallsetminus L$, 
and is already an embedding on $N$, 
we can find a smooth embedding $F\co M^3\to\C^3$ 
such that $F\restr_{M^3\smallsetminus L}$ 
is totally real 
and $F\restr_N=G\restr_N$ 
by using the relative $h$-principle. 
Thus $F\co M^3\to\C^3$ 
is the desired embedding with 
complex tangents forming $L$. 
\end{proof}

\begin{corollary}\label{cor:main}
Every knot or link in $S^3$ can be realized as the 
set of complex tangents of a smooth 
embedding $F\co S^3\to\C^3$.
\end{corollary}

\begin{acknowledgements}
The authors would like to thank Professor Minoru Yamamoto 
for many helpful comments 
on generic homotopy of stable maps. 
The second-named author has been 
supported in part by the Grant-in-Aid for Scientific Research (C), 
(No.~15K04880), Japan Society for the Promotion of Science. 
\end{acknowledgements}

%

\end{document}